\documentclass[11pt]{article}
\usepackage{amsmath, amsfonts,amsthm, amssymb}
\usepackage{multicol}
\usepackage{graphicx}
\usepackage{color}
\let\ol=\overline

\usepackage{epstopdf}
\usepackage{subfigure}
\usepackage[utf8]{inputenc}
\usepackage{hyperref}

\makeatletter
\let\cl@chapter\undefined
\makeatletter
\usepackage{cleveref}

\newtheorem{theorem}{Theorem}[section]
\newtheorem{lemma}[theorem]{Lemma}
\newtheorem{corollary}[theorem]{Corollary}
\newtheorem{proposition}[theorem]{Proposition}

\theoremstyle{definition}

\newtheorem{example}[theorem]{Example}
\newtheorem{remark}[theorem]{Remark}

\newcommand{\Aut}{{\rm Aut}}
\newcommand{\Sub}{{\rm Sub}}

\newcommand{\N}{{\mathbb N}}
\newcommand{\Z}{{\mathbb Z}}
\newcommand{\R}{{\mathbb R}}
\newcommand{\C}{{\mathbb C}}
\newcommand{\Q}{{\mathbb Q}}
\newcommand{\HH}{{\mathcal H}}
\newcommand{\D}{{\mathcal D}}
\newcommand{\TT}{\mathbb{T}}

\newcommand{\T}{\mathcal{T}}

\newcommand{\Id}{{\rm Id}}
\newcommand{\Homeo}{{\rm Homeo}}

\newcommand{\db}{\mathbf{d}}

\newcommand{\Inn}{{\rm Inn}}
\newcommand{\inn}{{\rm inn}}

\newcommand{\GL}{{\rm GL}}

\begin{document}

\title{Expansive Actions of Automorphisms of \\ Locally Compact Groups $G$ on Sub$_G$}

\author{Manoj B.\ Prajapati and Riddhi Shah}       

\maketitle
\begin{abstract}
For a locally compact metrizable group $G$, we consider the action of $\Aut(G)$ on $\Sub_G$, the space of all closed subgroups of $G$ 
endowed with the Chabauty topology. We study the structure of groups $G$ admitting automorphisms $T$ which act expansively on $\Sub_G$.
We show that such a group $G$ is necessarily totally disconnected, $T$ is expansive and that the contraction groups of $T$ and $T^{-1}$ are closed and their 
product is open in $G$; moreover, if $G$ is compact, then $G$ is finite. We also obtain the structure of the contraction group of such $T$. 
For the class of groups $G$ which are finite direct products of $\Q_p$ for distinct primes $p$, we show that 
$T\in\Aut(G)$ acts expansively on $\Sub_G$ if and only if $T$ is expansive. However, any higher dimensional $p$-adic vector space $\Q_p^n$, ($n\geq 2$), 
does not admit any automorphism which acts expansively on $\Sub_G$. 

\medskip

\medskip

\end{abstract}

\section{Introduction} 
\label{intro}

Let $(X,d)$ be a metric space and let $\Homeo(X)$ denote the space of homeomorphisms of $X$. 
Let $\Gamma$ be a group which acts on $X$ by homeomorphisms; i.e.\ there is a group homomorphism from $\Gamma$ to $\Homeo(X)$. 
The $\Gamma$-action on $X$ is said to be {\em expansive} if there exists $\delta>0$ which 
satisfies the following: If $x,y\in X$ with $x\neq y$, then $d(\gamma{(x)},\gamma(y))>\delta$ for some $\gamma\in\Gamma$. A homeomorphism 
$\varphi$ of $X$ is said to be {\em expansive} if the group $\{\varphi^n\}_{n\in\mathbb{Z}}$ acts expansively on $X$  
(equivalently, we say that the $\varphi$-action on $X$ is expansive). 

Let $G$ be a locally compact Hausdorff group with the identity $e$. An automorphism $T$ of $G$ is said to be {\em expansive} if $\cap_{n\in\Z}T^n(U)=\{e\}$ for 
some neighbourhood $U$ of $e$. 
If $T$ is expansive on $G$, then $G$ is metrizable and the above definition is equivalent to the one given in terms of the left invariant metric $d$ on $G$. 

 The notion of expansivity was introduced by Utz in \cite{Ut50} and studied by many in different contexts (see Bryant \cite{Br60}, Schmidt \cite{Sc95}, 
 Gl\"ockner and Raja \cite{GR17}, Shah \cite{Sh19} and references cited therein). 
 
We now recall the definition of $K$-contraction groups of $T$. Let $T\in\Aut(G)$ and let $K$ be a compact $T$-invariant subgroup of $G$. 
The group $C_K(T)=\{x\in G\mid T^n(x)K\to K \text{ in } G/K \mbox{ as } n\to\infty\}$ is called the $K$-contraction group of $T$. 
It is $T$-invariant and the group $C(T):=C_{\{e\}}(T)$ is called the {\em contraction group} of $T$. If $T$ is expansive and if $G$ is not discrete, 
by Theorem 2.9 of \cite{Sh19}, it follows that $T$ can not be distal and either $C(T)$ or $C(T^{-1})$ is nontrivial. 

We know that an automorphism of the $n$-dimensional torus $\TT^n$ (which belongs to $\GL(n,\Z)$) 
is expansive if and only if all its eigenvalues are of absolute value other than 1 (see e.g.\ \cite{KS89} or \cite{Sc95}). 
However, the unit circle does not admit any expansive homeomorphism \cite[Theorem 2.2.26]{AH94}. 
Here we are going to show that a large class of homeomorphisms on certain compact spaces are not expansive; namely, the 
homomorphisms of $\Sub_G$ arising from the automorphisms of almost connected locally compact infinite groups $G$.

Let $G$ be a locally compact (Hausdorff) topological group. 
Let $\Sub_{G}$ denote the set of all closed subgroups of $G$ equipped with the Chabauty topology \cite{Ch50}. 
Then $\Sub_G$ is compact and Hausdorff. It is metrizable if $G$ is so (see \cite{Ge18} and Chapter E of \cite{BP92} for more details). 
Let $\Aut(G)$ denote the group of automorphisms of $G$. There is a natural action of $\Aut(G)$ on $\Sub_G$; namely, 
$(T,H)\mapsto T(H)$, $T\in\Aut(G)$, $H\in\Sub_G$. For each $T\in\Aut(G)$, the map $H\mapsto T(H)$ defines a homeomorphism
 of $\Sub_G$ \cite[Proposition 2.1]{HK}, and the corresponding map from $\Aut(G)\to\Homeo(\Sub_G)$ is a group homomorphism. 
Here, we would like to study the action of an automorphism of a locally compact first countable (metrizable) group $G$ on $\Sub_{G}$ in terms of
 expansivity. 

Let $\Sub^a_{G}$ denote the space of all closed abelian subgroups of $G$. 
Note that $\Sub^a_{G}$ is closed in $\Sub_{G}$, and hence compact, and it is invariant under the action of $\Aut(G)$. 
Moreover, if $T\in\Aut(G)$ acts expansively on $\Sub_{G}$, then it acts expansively on $\Sub^a_{G}$. 
We show that a nontrivial connected Lie group $G$ does not have any automorphism which acts expansively on $\Sub^a_G$ (see \Cref{Lie-Group}). 
For a compact group $G$ and $T\in\Aut(G)$, we show that $T$ acts expansively on $\Sub_{G}$ if and only if $G$ is finite (see \Cref{Compact-Group}). 
As a consequence of these results, we get that any automorphism of a nontrivial connected locally compact group $G$ does not act expansively on 
$\Sub_G$ (see \Cref{ConnGroup}). 

 For a general locally compact group $G$ and $T\in\Aut(G)$, we show that if $T$ acts expansively on $\Sub_{G}$, then $G$ is totally 
 disconnected and $T$ acts expansively on $G$ (more generally, see \Cref{Totally-disconnected}). Let $\Q_p$ denote the $p$-adic field for a prime $p$.
 It is locally compact and totally disconnected and, it is a topological group with addition as the group operation. For 
 $G=\Q_{p_1}\times\cdots\times\Q_{p_n}$, a finite direct product, where $p_1, \ldots, p_n$ are distinct primes, we show that $T$  
 acts expansively on $\Sub_G$ if and only if $T$ is expansive (see \Cref{n-primes}). However, we show that 
any higher dimensional $p$-adic vector space $\Q_p^n$, ($n\geq 2$), does not admit any automorphism 
which acts expansively on $\Sub_{\Q_p^n}$ (see \Cref{p-adic}). Using these results and a structure theorem for totally disconnected locally compact contraction groups 
obtained by Gl\"ockner and Willis in \cite{GW10}, 
we prove that if an automorphism of $G$ acts expansively on $\Sub_G$, then its contraction group is either trivial or a finite direct product of 
 $\Q_p$ for distinct primes $p$ (see \Cref{main} and \Cref{rem3}). 

For many groups $G$, the space $\Sub_G$ and $\Sub^a_G$ have been identified 
(see e.g.\  Pourezza and Hubbard \cite{PH79}, Bridson, de la Harpe and Kleptsyn \cite{BHK09} and also Baik and Clavier \cite{BC16}). 
Since the homeomorphisms of $\Sub_G$ arising from the action of $\Aut(G)$ form a large 
subclass of it, it is significant to study the expansivity of such homeomorphisms of $\Sub_G$. 

A homeomorphism $\varphi$ of a topological space $X$ is said to be {\em distal} if for every pair of distinct elements $x,y\in X$, the closure of the 
double orbit $\{(\varphi^n(x), \varphi^n(y))\mid n\in\Z\}$ does not intersect the diagonal $\{(d,d)\mid d\in X\}$ in $X\times X$. 
In case $X$ is a compact metric space with a metric $d$, $\varphi$ is distal if $\inf_{n\in\Z} d(\varphi^n(x),\varphi^n(y))>0$, $x,y\in X$.
Note that for homeomorphisms of compact infinite metric spaces, distality and expansivity are opposite phenomena \cite[Theorem 2]{Br60}. 
Shah and Yadav in \cite{SY19} have discussed the distality of the actions of automorphisms on $\Sub_G$ and $\Sub^a_G$.  
The study of expansivity of these actions in the current paper contributes to and enhances the understanding of  the dynamics of actions of 
automorphisms of $G$ on $\Sub_G$. 

 Throughout, $G$ is a locally compact first countable (metrizable) topological group. For a closed subgroup $H$, let $H^0$ denote the 
 connected component of the identity $e$ in $H$. For $T\in\Aut(G)$, we say that a subset $A$ of $G$ is $T$-invariant if $T(A)=A$.

\section{Basic Results on Expansive Actions and Sub$_{\boldsymbol{G}}$} 
\label{sec:1}

Given a locally compact (metrizable) group, the Chabauty topology on $\Sub_{G}$ was introduced by Claude Chabauty in \cite{Ch50}. 
A sub-basis of the Chabauty topology on $\Sub_{G}$ is given by the sets of the following form $O_{K}=\{A\in\Sub_{G}\mid A\cap K=\emptyset\}$, where 
$K$ is a compact subset of $G$, and $O_{U}=\{A\in\mathrm{ Sub}_{G}\mid  A\cap U\neq\emptyset\}$, where $U$ is an open subset of $G$. 

Any closed subgroup of $\R$ is either a discrete group generated by a real number or the whole group $\R$, and 
$\Sub_{\R}$ is homeomorphic to $[0,\infty]$ with a compact topology. 
Any closed subgroup of $\Z$ is of the form $n\Z$ for some $n\in\N\cup\{0\}$, 
and $\Sub_{\Z}$ is homeomorphic to $\{\frac{1}{n}\mid n\in\N\}\cup\{0\}$.
The space $\Sub_{\R^2}$ is homeomorphic to $\mathbb{S}^4$ (see \cite{PH79}).
Note that the space $\Sub_{\R^n}$ is simply connected for all $n\in\N$ \cite[Theorem 1.3]{Kl09}. 

We first recall some known results which will be useful. 

\begin{lemma} $\cite[\mbox{ Corollary }5.22\ \&\ \mbox{Theorem }5.26]{Wa82}$  \label{e-basic}
Let $(X,d)$ be a compact metric space. Then the following hold for homeomorphisms of $X$: 
\begin{enumerate}
\item[$(1)$] Expansivity is a topological conjugacy invariant. 
\item[$(2)$] Expansivity of a homeomorphism is independent of the metric chosen as long as the metric induces the topology of X.
\end{enumerate}
Moreover, the following hold for any homeomorphism $\varphi$ of $X$: 
\begin{enumerate}
\item[$(3)$] For every $n\in\Z\setminus\{0\}$, if $\varphi$ is expansive, then $\varphi^n$ has only finitely many fixed points. 
\item[$(4)$] $\varphi^n$ is expansive for some $n\in\Z\setminus\{0\}$ if and only if $\varphi^n$ is expansive for all $n\in\Z\setminus\{0\}$.
\item[$(5)$] If $\varphi$ is expansive and $Y$ is a closed $\varphi$-invariant subset of $X$, then $\varphi|_{Y}$ is also expansive.
\end{enumerate}
\end{lemma}

\begin{remark} \label{rem1} Note that $(4)$ and $(5)$ of \Cref{e-basic} hold for any automorphism of a  locally compact $($not necessarily 
compact$)$ group and $(5)$ of $\Cref{e-basic}$ also holds for any $($not necessarily compact$)$ metric space. 
\end{remark}

Recall that there is a natural group action of $\Aut(G)$, the group of automorphisms of $G$, on 
$\Sub_G$ defined as follows:
$$\Aut(G)\times\Sub_G\to\Sub_G;\ (T,H)\mapsto T(H), T\in\Aut(G), H\in\Sub_G.$$
The map $H\mapsto T(H)$ is a homeomrphism of $\Sub_G$ for each $T\in\Aut(G)$, and the corresponding map from $\Aut(G)$ to $\Homeo(\Sub_G)$ is a homomorphism. 

The following elementary lemma will be useful and it is easy to prove. However we give a proof for the sake of completion. 

\begin{lemma}\label{subg-infinite}
If $G$ is an infinite locally compact group, then $\Sub^a_{G}$ is also infinite. 
In particular, for  an infinite closed subset $\HH$ of $\Sub_{G}$, the action of the trivial map $\Id$ on $\HH$ is not expansive.
\end{lemma}

\begin{proof}
If $G$ has infinitely many elements of finite order, then $\Sub^a_{G}$ is infinite. 
Suppose $G$ has finitely many elements of finite order. Then $G$ has an 
element say, $x$ which does not have finite order.  
Then the cyclic subgroup $G_x$ generated by $x$ is either discrete and isomorphic to $\Z$ or its closure $\overline{G}_x$ is a 
compact group \cite[Theorem 19]{Mo77}. 
In the first case, there are infinitely many closed subgroups; namely, $G_{x^n}$ generated by $x^n$, $n\in\N$. 
In the second case $\overline{G}_x$ is a compact abelian infinite group. 
Then its Pontryagin dual is an infinite discrete abelian group, so has infinitely many subgroups. 
By Pontryagin duality, it follows that $\overline{G}_x$ has infinitely many closed subgroups. 
Therefore $\Sub_{\overline{G}_x}$, and hence $\Sub^a_{G}$ is infinite. 

As the $T$-action on $\HH$ has infinitely many fixed points, by \Cref{e-basic}\,$(3)$, the $\mathrm{Id}$-action on $\HH$ is not expansive.
\end{proof}

The following lemma shows that the expansivity carries over to the quotients modulo closed invariant subgroups. 

\begin{lemma}\label{subg-q}
Let $G$ be a locally compact group, $T\in\Aut(G)$ and let $H$ be a  $T$-invariant closed normal subgroup of $G$. 
Let $\ol{T}\in\Aut(G/H)$ be the corresponding map defined as $\overline{T}(gH)=T(g)H$, for all $g\in G$. 
If the $T$-action on $\Sub_{G}$ is expansive, then the $T$-action on $\Sub_H$ as well as the $\ol{T}$-action on $\Sub_{G/H}$ are also expansive.
\end{lemma}

\begin{proof}
As $T(H)=H$, $\Sub_H$ is $T$-invariant. 
Also, $\HH=\{K\in\Sub_G: H\subset K\}$ is a $T$-invariant closed subset of $\Sub_G$. 
Therefore, by \Cref{e-basic}\,$(5)$, the $T$-action on both $\Sub_H$ and $\HH$ is expansive.  
Let $\pi:\HH\to\Sub_{G/H}$  be defined as $\pi({K})=K/H$. Then $\pi$ is a homeomorphism \cite[Proposition 2]{Sm71}. 
By \Cref{e-basic}\,(1), $\pi T\pi^{-1}$ acts expansively on $\Sub_{G/H}$.  As the $\pi T\pi^{-1}$-action on $\Sub_{G/H}$ 
is the same as the $\ol{T}$-action on $\Sub_{G/H}$, the $\ol{T}$-action on $\Sub_{G/H}$ is expansive.
\end{proof}

The converse of the above lemma does not hold as illustrated by \Cref{ex}. 

We say that a Hausdorff topological space $K$ acts continuously on a topological space $X$ if there exists 
a continuous map from $K\times X$ to $X$, $(k,x)\mapsto kx$, $k\in K$, $x\in X$. 
In particular, each $k\in K$ defines a continuous map on $X$.

The following lemma is well-known and easy to prove. 

\begin{lemma} \label{compact-action}
Let $(X, d)$ be a metric space and let $K$ be a compact space which acts continuously on $X$. Then  
this action of $K$ on $X$ is uniformly equicontinuous; $($i.e.\ given $\epsilon>0$, there exists $\delta>0$, for any pair of elements 
$x,y\in X$ with $d(x,y)<\delta$, we have $d(kx,ky)<\epsilon$ for all $k\in K)$. 
\end{lemma}

Let $H$ be (Hausdorff) topological group. We say that $H$ 
acts continuously on a topological space X by homeomorphisms if there exists a homomorphism 
$\rho : H\to \Homeo(X)$ such that the corresponding map $H \times X \to X$ given by $(h, x)\mapsto \rho(h)(x)$ is continuous. 
The following lemma will be useful and it can be easily proven by using \Cref{compact-action}. 
The lemma will apply in particular to the action of $H=\Aut(G)$ on $X=\Sub_G$, when $G$ is a connected Lie group or a $p$-adic vector space. 

\begin{lemma} \label{GroupAction}  Let $X$ be a metric space and let $H$ be a topological group acting continuously on $X$ by homeomorphisms. 
Let $\varphi, \psi\in H$ be such that $\varphi\psi=\psi\varphi$ and $\psi$ is contained in a compact subgroup of $H$. 
Then $\varphi$ is expansive if and only if $\varphi\psi$ is expansive. 
\end{lemma}

Note that $\Aut(G)$ with the modified compact-open topology is a topological group \cite[9.17]{St06}, and by 
Lemma 2.4 of \cite{SY19} it acts continuously on $\Sub_G$ by homeomorphisms. 
Hence \Cref{GroupAction} holds for the case when $X=\Sub_G$ and $H=\Aut(G)$ with the modified compact-open topology, 
where $G$ is a locally compact metrizable group. 

\section{Actions of Automorphisms of Lie Groups $\boldsymbol{G}$ on Sub$_{\boldsymbol{G}}$} 
\label{sec:2}

In this section we prove that a connected Lie group $G$ does not admit any automorphism which acts expansively on $\Sub^a_G$. 
We also prove that a compact group $G$ admits an automorphism which acts expansively on $\Sub_G$ if and only if $G$ is finite.  
As a consequence of these results, we show that any connected locally compact group $G$ does not admit an automorphism which acts 
expansively on $\Sub_G$. 

\begin{theorem}\label{Lie-Group}
Let $G$ be a nontrivial connected Lie group and let $T\in\Aut(G)$. Then the $T$-action on $\Sub^a_G$ is not expansive.
\end{theorem}

\begin{proof}
\noindent\textbf{Case I: } Suppose $G\cong\R$. 
If $T=\pm\,\Id$, then $T$ acts trivially on $\Sub_\R$. 
It follows by \Cref{subg-infinite} that the $T$-action on $\Sub_\R$ is not expansive. 
Now suppose $T=\alpha\,\Id$ for some $\alpha\in\R^*$. 
Replacing $T$ by $T^{-1}$ if necessary, we may assume that $0<\alpha<1$. 
Let $K=[1/2, 3/2]$. Then $K$ is compact and it acts continuously on $\Sub_\R$ through the action given as follows: 
$(k, H)\mapsto kH$, $k\in K$, $H\in\Sub_{\R}$. Let $\db$ be a metric on $\Sub_{\R}$. 

Let $\epsilon>0$ be fixed. 
By \Cref{compact-action}, there exists $\delta>0$ such that if $H_1,H_2\in \Sub_\R$ with $\db(H_1,H_2)<\delta$, 
then $\db(kH_1,kH_2)<\epsilon/2$, $k\in K$. 

We know that $T^n(\Z)=\alpha^n\Z\to\R$ 
and $T^{-n}(\Z)=\alpha^{-n}\Z\to\{0\}$ as $n\to\infty$. 
There exists $n_0\in\N$ such that
$$
{\rm max}\{\db(\alpha^n\Z,\R),~\db(\alpha^{-n}\Z,\{0\})\}<\delta \text{ for all } n\geq n_0.$$
This implies that
$$
{\rm max}\{\db(k\alpha^n\Z,\R),\db(k\alpha^{-n}\Z,\{0\})\}<\epsilon/2 \text{ for all } n\geq n_0 \text{ and } k\in K.$$
Now for each $k\in K$ and $n\in\N$ with $n\geq n_0$, we have 
$$
\db(\alpha^n\Z,\alpha^nk\Z)\leq \db(\alpha^n\Z,\R)+\db(\alpha^nk\Z,\R)<\epsilon/2+\epsilon/2=\epsilon.$$
Similarly, we have $\db(\alpha^{-n}\Z,\alpha^{-n}k\Z)<\epsilon$ for all $n\geq n_0$ and $k\in K$. Since the ${T}^n$-action on 
$\Sub_{\R}$ is continuous for all $n\in\Z$, there exists $\delta'>0$ with $0<\delta'<\delta$ such that for any $k\in\R$ if 
$\db(k\Z,\Z)<\delta'$, then $\db(\alpha^n\Z,\alpha^nk\Z)<\epsilon$ for all $|n|\leq n_0$
It is easy to see that for $k_t\in K$,  if $k_t\to 1$, then $k_t\Z\to\Z$. 
This implies that there exists $k'\in K$, such that $\db(k'\Z,\Z)<\delta'$. 
Therefore, $\db(\alpha^n\Z,\alpha^n k'\Z)<\epsilon$ for all $n\in \Z$. 
Since $\epsilon>0$ is arbitrary and $T=\alpha\,\Id$, the $T$-action on $\Sub_\R$ is not expansive. 
 
Suppose $G\cong\R^n$ for some $n\geq 2$ and $T\in{\rm GL}(n,\R)$. 
Suppose $T$ has a real eigenvalue. 
Then we get a one dimensional $T$-invariant space $V\cong\R$ such that $T|_V=\alpha\,\Id$, for some $\alpha\ne 0$. 
From above, the $T$-action on $\Sub_{V}$, and hence on $\Sub_{\R^n}$ is not expansive. 
Now suppose all the eigenvalues of $T$ are complex. 
There exists a two-dimensional subspace $W$ which is $T$-invariant. 
Then there exist $t>0$ and a rotation map $S\in\GL(2,\R)$ such that $T|_W=tASA^{-1}$ for some $A\in\GL(2,\R)$. 
Observe that $ASA^{-1}$ is contained in a compact subgroup of $\GL(2,\R)$. 
As $W\cong\R^2$, from the above argument, $t\,\Id$-action on $\Sub_{W}$ is not expansive. As $\Aut(W)\cong \GL(2,\R)$, 
its topology coincides with the compact-open topology as well as with the modified compact-open topology, and it acts
continuously on $\Sub_W$ by homeomorphisms; (one can also directly show this by using the criteria for convergence of sequences in $\Sub_G$ 
as in \cite[Proposition E.1.2, pp 161]{BP92}). 
By \Cref{GroupAction}, the $T$-action on $\Sub_{W}$, and hence on $\Sub_{\R^n}$ is also not expansive. 
Note that $\Sub_{\R^n}=\Sub^a_{\R^n}$.

\smallskip
\noindent\textbf{Case II: } Suppose $G$ is not isomorphic to $\R^n$ for any $n\in\N$.
Let $R$ be the radical of $G$, i.e.\ $R$ is the largest connected closed solvable normal subgroup. 
Suppose $R\neq\{e\}$. 
Let $R_0=R$, $R_1=\overline{[R,R]}$, $R_{m+1}=\overline{[R_m,R_m]}$ for all $m\geq 1$. 
These are closed connected characteristic subgroups of $G$. 
As $R$ is solvable, there exists $k\in\N\cup\{0\}$ such that $R_k\ne\{e\}$ and $R_{k+1}=\{e\}$. 
As $R_k$ is a nontrivial $T$-invariant closed connected abelian Lie subgroup, $R_k\cong\R^n\times\TT^m$ 
for some $m,n\in\N\cup\{0\}$ such that $m+n$ is nonzero. 
Suppose $m\ne 0$. Then $\TT^m$ is characteristic in $G$, and hence it is $T$-invariant. 
Let $K$ be the set of all roots of unity in $\TT^m$, i.e.\ $K=\cup_{i\in\N}{K_{i}}$, where $K_i=\{x\in\TT^m\mid x^i=e\}$. 
Note that $K_{i}$ is a nontrivial proper closed $T$-invariant subgroup of $\TT^m$, $i\in\N$, and $K_i\ne K_j$ if $i,j\in\N$ are distinct. 
By \Cref{e-basic}\,$(3)$, we get that the $T$-action on $\Sub_{\TT^m}$, and hence on $\Sub^a_G$ is not expansive.  
If $m=0$, then $R_k\cong\R^n$ and from Case I, the $T$-action on $R_k$, and hence on $\Sub^a_G$ is not expansive. 
 
Suppose $R=\{e\}$. Then $G$ is semisimple. 
Then the group $\Inn(G)$ of all inner automorphisms of $G$ is a subgroup of finite index in $\Aut(G)$ (see section 1 in \cite{Da18}). 
There exist $x\in G$ and $n\in\N$ such that $T^n=\inn(x)$; i.e.\ $T^n(g)=xgx^{-1}$, $g\in G$. 
Suppose $x^m\in Z(G)$ for some $m\in\N$. Then $T^{nm}=\Id$ and by \Cref{subg-infinite} and \Cref{e-basic}\,$(3)$, the 
$T$-action on $\Sub^{a}_{G}$ is not expansive. Suppose $x^k\not\in Z(G)$ for every $k\in\N$. 
Let $H_x$ be the closure of the cyclic subgroup generated by $x$. Then $H_x$ is an infinite closed abelian $T$-invariant subgroup. 
As $T$ acts trivially on $H_x$, by \Cref{subg-infinite}, the $T$-action on $\Sub_{H_x}$, and hence on $\Sub^a_G$ is not expansive. 
\end{proof}

Let $X$ be a topological space and let $f$ be a continuous self map on $X$. 
A point $x\in X$ is said to be {\em periodic} if $f^n(x)=x$ for some $n\in\N$. The set ${\rm Per}(f)$ 
denotes the set of all periodic points of $X$. 

Let $G$ be a compact group and let $T\in\Aut(G)$. 
We say that $(G,T)$ satisfies the {\em descending chain condition} if for every decreasing sequence 
$\{G_n\}$ of closed $T$-invariant subgroups in $G$, there exists $k\in\N$ such that $G_n=G_k$ for all $n\geq k$.
Note that if $T$ acts expansively on $\Sub_G$, then by \Cref{e-basic}\,$(3)$, $(G,T)$ satisfies the descending chain condition. 

\begin{theorem}\label{Compact-Group}
Let $G$ be a compact metrizable group and let $T\in\Aut(G)$. 
Then the $T$-action on $\Sub_{G}$ is expansive if and only if $G$ is finite.
\end{theorem}

\begin{proof} The `if' statement is obvious. Now suppose $G$ is an infinite compact group. We show that the $T$-action on $\Sub_G$ is not expansive. 
As noted above, if $(G,T)$ does not satisfy the descending chain condition, then by \Cref{e-basic}\,$(3)$, the $T$-action on $\Sub_{G}$ is not expansive. 
Now suppose $(G,T)$ satisfies the descending chain condition.  
By Proposition 3.5 of \cite{Sc95}, there exists a compact normal $T$-invariant subgroup $H$ of $G$ such that $G/H$ is Lie group and 
the $T$-action on $H$ is ergodic. Suppose $H$ is trivial. Then $G$ is a compact Lie group with finitely many connected components. 
As $G$ is infinite, $G^0$ is nontrivial and by \Cref{Lie-Group}, the $T$-action on $\Sub_{G^0}$, and hence on $\Sub_G$ is not expansive. 

Now suppose $H$ is nontrivial. Then it is infinite. Let ${\rm Per}(T)$ be the set of all periodic points of $T$ in $H$. 
As $T|_H$ is ergodic and $(H,T|_H)$ also satisfies the descending chain condition, by Theorem 7.5 of \cite{KS89}, ${\rm Per}(T)$ is dense in $H$. 
Here ${\rm Per}(T)=\cup_{n\geq 1}{P_{n}}$, where $P_n=\{x\in H\mid T^n(x)=x\}$ and each $P_n$ is a $T$-invariant closed subgroup of $H$, $n\in\N$. 
Since $T$ is ergodic, so is $T^n$, and hence $T^n\ne\Id$ and $P_n\neq {\rm Per}(T)$ for every $n\in\N$. 
By \Cref{e-basic}\,$(3)$, the $T$-action on $\Sub_{H}$, and hence on $\Sub_G$ is not expansive.
 \end{proof}

\begin{corollary}\label{ConnGroup}
A nontrivial connected locally compact group $G$ does not admit any automorphism which acts expansively on $\Sub_{G}$.
\end{corollary}

\begin{proof}
Since $G$ is connected, it admits the largest compact normal subgroup $K$ such that $G/K$ is a Lie group. 
Let $T\in\Aut(G)$. Here, $K$ is characteristic in $G$ and, in particular, it is $T$-invariant. 
If $K$ is infinite, then by \Cref{Compact-Group}, the $T|_{K}$-action on $\Sub_{K}$ is not expansive. 
If $K$ is finite, then $G$ is a Lie group and by \Cref{Lie-Group}, the $T$-action on $\Sub^a_G$, and hence on $\Sub_G$ is not expansive. 
\end{proof}

\section{Expansive Actions on Sub$_{\boldsymbol{G}}$ of Automorphisms of Locally Compact Groups $\boldsymbol{G}$} 
\label{sec:3}

In this section we prove certain results for the structure of locally compact groups $G$ which admits an automorphism $T$ that acts expansively on 
$\Sub_{G}$. We show that such $G$ is necessarily totally disconnected, $T$ is expansive (on $G$), and the contraction group of $T$ is 
either trivial or a finite direct product of $\Q_p$, for some distinct primes $p$.  We show that any automorphism of $G$, where $G$ is a finite 
direct product of $\Q_p$ for distinct primes $p$, is expansive if and only if it acts expansively on 
 $\Sub_G$. We also show that any higher dimensional $p$-adic vector space $\Q_p^n$, ($n\geq 2$), does not admit any automorphism
 which acts expansively on $\Sub_{\Q_p^n}$. 
 
For $T\in\Aut(G)$, let $M(T)=\{x\in G\mid \{T^n(x)\}_{n\in\Z}\mbox{ is relatively compact}\}$.
Then $M(T)$ is a $T$-invariant subgroup of $G$. If $G$ is totally disconnected, then $M(T)$ is closed 
\cite[Proposition 3]{Wi94} (see also Remark 3.1 in \cite{BW04}).

\begin{theorem}\label{Totally-disconnected}
Let $G$ be a locally compact group $G$ and let $T\in\Aut(G)$ be such that 
the $T$-action on $\Sub_{G}$ is expansive. Then the following  hold: 
\begin{enumerate}
\item[$(1)$] $G$ is totally disconnected.
\item[$(2)$]  $T$ is expansive.
\item[$(3)$]  $C(T)$ and $C(T^{-1})$ are closed, and $C(T)C(T^{-1})$ is open.
\item[$(4)$] $M(T)$ is finite.
\end{enumerate}
\end{theorem}

\begin{proof}
\noindent $(1):$ As $T|_{G^0}$ is expansive, $G^0$ is trivial by \Cref{ConnGroup}. 
Hence $G$ is totally disconnected.

\smallskip
\noindent $(2):$ If $G$ is discrete, then $T$ is expansive. 
Suppose $G$ is not discrete and suppose $T$ is not expansive. 
There exists $\{U_n\mid n\in\N\}$, a neighbourhood basis of the identity $e$ consisting of compact open subgroups such that 
$K_n=\cap_{k\in\Z}T^k(U_n)\neq\{e\}$, $n\in\N$. Note that $K_n$ is a nontrivial closed $T$-invariant subgroup of $G$, $n\in\N$.  
By \Cref{e-basic}\,$(3)$, the $T$-action on $\Sub_G$ is also not expansive.

\smallskip
\noindent $(3):$ By Corollaries 3.27 and 3.30 of \cite{BW04}, there exists a compact $T$-invariant subgroup $K=\overline{C(T)}\cap\overline{C(T^{-1})}$ 
such that $\overline{C(T)}=C_K(T)=KC(T)$. Since the $T$-action on $\Sub_G$ is expansive and $T(K)=K$, $T|_K$ acts expansively on $\Sub_K$. 
Then $K$ is finite by \Cref{Compact-Group}. By Lemma 3.31\,$(2)$ of \cite{BW04}, we get that $C(T)\cap M(T)$ is dense in $K$. 
Since $K$ is finite and $T$-invariant, and since $M(T)$ is closed and $T$-invariant, 
we have that $C(T|_{M(T)})=C(T)\cap M(T)=C(T)\cap K=C(T|_K)=\{e\}$. Therefore $K=\{e\}$, and hence $C(T)$ is closed. 
Replacing $T$ by $T^{-1}$, we get that $C(T^{-1})\cap M(T)=\{e\}$ and $C(T^{-1})$ is also closed. 
By  Lemma 1.1\,(d) of \cite{GR17}, $C(T)C(T^{-1})$ is open in $G$ and (3) holds. 

\smallskip
\noindent $(4):$ As observed above, $M(T)$ is a closed $T$-invariant subgroup. As shown in the proof of $(3)$, 
$C(T|_{M(T)})=M(T)\cap C(T)=\{e\}=M(T)\cap C(T^{-1})=C(T^{-1}|_{M(T)})$ 
(see also Theorem 3.32 in \cite{BW04}). By $(3)$, $\{e\}$ is open in $M(T)$, and hence $M(T)$ is discrete. 
Now for every $x\in M(T)$, the $T$-orbit of $x$ is relatively compact, and hence finite. Therefore, 
$M(T)={\rm Per}(T)=\cup_{n\in\N}P'_n$, where $P'_n=\{x\in M(T)\mid T^{n!}(x)=x\}$. Each $P'_n$ is a closed $T$-invariant subgroup and 
$P'_n\subset P'_{n+1}$, $n\in\N$. If possible, suppose $P'_n\ne P'_{n+1}$ for infinitely many $n$. By \Cref{e-basic}\,$(3)$, the $T$-action on 
$\Sub_{M(T)}$ is not expansive, which leads to a contradiction. Therefore, $M(T)=P'_n$ for some $n$, and hence $T^{n!}=\Id$.
By \Cref{subg-infinite}, $M(T)$ is finite.
\end{proof}

\begin{corollary}\label{discrete}
Let $G$ be a locally compact group $G$ and let $T\in\Aut(G)$. 
If the $T$-action on $G$ is distal and the $T$-action on $\Sub_{G}$ is expansive, then $G$ is discrete.
\end{corollary}

\begin{proof}
If the $T$-action on $\Sub_G$ is expansive, then by \Cref{Totally-disconnected}\,$(2)$, we get that $T$ is expansive. 
Moreover, if the $T$-action on $G$ is also distal, then $C(T)=C(T^{-1})=\{e\}$, and by \Cref{Totally-disconnected}\,$(3)$, $G$ is discrete 
(see also \cite[Theorem 2.9]{Sh19}).
\end{proof}

From \Cref{Totally-disconnected}, we know that if $G$ admits an automorphism which acts expansively on $\Sub_G$, it is necessarily totally disconnected. 
\Cref{p-adic} below shows that $\Q_p$ admits such an automorphism, while none of the $\Q_p^n$, $n\geq 2$, does.

Recall that for a fixed prime $p$, and $a/b\in\Q\setminus\{0\}$, the $p$-adic absolute value of $a/b$ is defined as $|a/b|_p=p^{-n}$ if $a/b=p^nc/d$ for some $n\in\Z$, 
where $c$ and $d$ are co-prime to $p$, and $|0|_p=0$. 
Observe that $|\cdot|_p$ defines a norm on $\Q$, and $\Q_p$ is a completion of $\Q$ with respect to the metric induced by 
this norm. Moreover, $|\cdot|_p$ extends canonically to $\Q_p$ such that $|xy|_p=|x|_p|y|_p$, $x,y\in\Q_p$. 
It is a locally compact totally disconnected field of characteristic zero, and $\Q_p^n$ is the $n$-dimensional $p$-adic vector space. 
Here, we consider $\Q_p^n$ as a group with addition as the operation. Any automorphism of $\Q_p^n$ belongs to $\GL(n,\Q_p)$. 
Any automorphism $T$ of $\Q_p$ is of the form $T(x)=qx$, $x\in\Q_p$, where $q\in\Q_p^*$, i.e.\ $q=p^mz$, for some $m\in\Z$ 
and some $z\in \Z_p^*=\{x\in\Q_p\mid |x|_p=1\}$. Note that $\Z_p^*$ is a compact multiplicative group. Observe that such a $T$ is 
expansive only if  $q\not\in\Z_p^*$, i.e.\ $m\ne 0$.  

\begin{proposition}\label{p-adic}
Let $T\in\GL(n,\Q_p)$ for some $n\in\N$.
\begin{enumerate}
\item For $n=1$, the $T$-action on $\Sub_{\Q_p}$ is expansive if only if $T$ is expansive.  
\item For $n\geq 2$, the $T$-action on $\Sub_{\Q_p^n}$ is not expansive. That is, $\Q_p^n$, $(n\geq 2)$, does not admit any automorphism which 
acts expansively on $\Sub_{\Q_p^n}$. 
\end{enumerate}
\end{proposition}

\begin{proof}
$(1):$ The `only if' statement follows from  \Cref{Totally-disconnected}\,$(1)$. 
Now suppose $T$ acts expansively on $\Q_p$. Then $T=p^ma\,\Id$, for some $m\in\Z\setminus\{0\}$ for $a\in \Z_p^*$, i.e.\ $|a|_p=1$. 
Note that any closed subgroup of $\Q_p$ is equal to $\{0\}$, $\Q_p$ or $p^n\Z_p$ for some $n\in\Z$, where $\Z_p=\{x\in\Q_p\mid |x|_p\leq 1\}$. 
Observe that the $T$-action on $\Sub_{\Q_p}$ is the same as the $bT$-action on $\Sub_{\Q_p}$ for any $b\in\Z_p^*$. 
By \Cref{e-basic}\,$(4)$, without loss of any generality we may assume that $T=p\,\Id$. 

If possible, suppose the $T$-action on $\Sub_{\Q_p}$ is not expansive. 
Let $\db$ be a metric on $\Sub_{\Q_p}$. 
Let $k_0$ be such that 
$$
\frac{1}{k_0}<\db(H_1,H_2) \mbox{ for all }H_1,H_2\in\{\{0\},\Z_p,\Q_p\}\mbox{ with }H_1\ne H_2.$$
Let $k\in\N$ be such that $k>k_0$. Since the $T$-action on $\Sub_{\Q_p}$ is not expansive, for any such $k$, there exist $i_k,j_k\in\Z$ with 
$i_k< j_k$ such that for all $n\in\Z$,
$$
\db(T^n(p^{i_k}\Z_p),T^n(p^{j_k}\Z_p))<\frac{1}{k}.$$
Putting $n=-i_k$, we get for $l_k=j_k-i_k$, $k\in\N$ with $k> k_0$ that 
$$
\db(\Z_p,p^{l_k}\Z_p)<\frac{1}{k},$$
and hence, $p^{l_k}\Z_p\to\Z_p$. 
But as $l_k\in\N$, passing to a subsequence if necessary, we get that either $l_k\to\infty$ and $p^{l_k}\Z_p\to \{0\}$ in $\Sub_{\Q_p}$, or  
$l_k=l\in\N$ and $p^{l_k}\Z_p= p^l\Z_p\ne\Z_p$, $k\in\N$. 
In either case, we get  a contradiction. 
Hence the $T$-action on $\Sub_{\Q_p}$ is  expansive.

\smallskip
\noindent $(2):$ Let $n\geq 2$ and let $T\in\GL(n,\Q_p)$.
We show that the $T$-action on $\Sub_{\Q_p^n}$ is not expansive. 
By 3.3 of \cite{Wa84}, there exists $m\in\N$ such that $T^m=SAD$, where $S$, $A$ and $D$ are semisimple, unipotent and 
diagonal matrix respectively, $A$, $S$ and $D$ commute with each other and $A$ as well as $S$ generate a relatively compact group. 
Take $U=SA$.
Then $T^m=UD=DU$ where $U$ generates a compact group and $D(x_1,\ldots,x_n)=(p^{k_1}x_1,\ldots,p^{k_n}x_n)$ with $k_i\in\Z$, $1\leq i\leq n$. 
By \Cref{e-basic}\,$(4)$, it is enough to show that the $T^m$-action on $\Sub_{\Q_p^n}$ is not expansive. Note that $\GL(n,\Q_p)$ is a metrizable 
topological group. Using the criteria for convergence of sequences in $\Sub_G$ for a metrizable group $G$ as in \cite{BP92}, it is easy to see that 
$\GL(n,\Q_p)$ acts continuously on $\Sub_{\Q_p^n}$ by homeomorphisms. By \Cref{GroupAction}, it is enough to show that the $D$-action on 
$\Sub_{\Q_p^n}$ is not expansive. 

Let $H=\{(x_1,\dots, x_n)\mid x_i=0 \mbox{ for all } i\geq 3\}$. $H\cong\Q_p^2$ is a $D$-invariant closed subgroup 
of $\Q_p^n$, and it enough to show that the $D|_H$-action on $\Sub_H$ is not expansive. 
Hence we may assume that $n=2$, i.e.\ $D(x_1,x_2)=(p^{k_1}x_1,p^{k_2}x_2)$ for some $k_1,k_2\in\Z$.  

Let $\HH$ be the collection of all closed subgroups of $\Q_p^2$ of the form $(a,b)\Q_p$ with $(a,b)\in\Q_p^2\setminus\{(0,0)\}$ (Here,  
$\HH$ is the set of $p$-adic lines in $\Q_p^2$). It is clear that $\HH$ is a $D$-invariant closed subset of $\Sub_{\Q_p^2}$.
We will show that the $D$-action on $\HH$ is not expansive. 
Since the $D$-action on $\HH$ is the same as that of $p^{-k_2}D$, we may replace $D$ by $p^{-k_2}D$ and assume that $k_2=0$. 
Now if $k_1=0$, then $D=\Id$ and by \Cref{subg-infinite}, the $D$-action on $\HH\subset\Sub_{\Q_p^2}$ is not expansive. 
Suppose $k_1\ne 0$. Since $D=D_1^{k_1}$, where $D_1(x,y)=(px_1,x_2)$, $x_1,x_2\in\Q_p$, and $D_1$ also keeps $\HH$ invariant,
 by \Cref{e-basic}\,$(4)$ we may replace $D$ by $D_1$ and assume that $D(x_1,x_2)=(px_1, x_2)$, $x_1,x_2\in\Q_p$.
Observe that $D((0,1)\Q_p)=(0,1)\Q_p$ and  that $D^n((a,1)\Q_p)\to (0,1)\Q_p$ as $n\to\infty$, for all $a\in\Q_p$. 
Moreover, the set $\{(a,1)\Q_p\mid a\in\Q_p\}$ is uncountable as $\Q_p$ is so. 
By Theorem 1 of \cite{Re70}, the $D$-action on $\HH$ is not expansive. 
Therefore, the $D$-action on $\Sub_{\Q_p^2}$ is not expansive.
\end{proof}

\begin{remark} \label{rem2}
Since the metric on $\Q_p^n$ (naturally defined by the $p$-adic norm) is proper, \Cref{p-adic} can also be proven using the explicit metric 
on $\Sub_{\Q_p^n}$ given as in \cite{Ge18}. As mentioned in \cite{Ge18}, this metric on $\Sub_G$ for a locally compact metrizable group $G$ has 
been suggested by Biringer in \cite{Bi18}.  
\end{remark}

The following example illustrates that the converse of \Cref{subg-q} does not hold in general. 

\begin{example}\label{ex}
Let $T=p\,\Id$ in $\GL(2,\Q_p)$.  
If $H=\Q_p\times\{0\}$, then $H$ is $T$-invariant and by \Cref{p-adic}\,(1), the $T$-action on both 
$\Sub_{H}$ and $\Sub_{\Q_p^2/H}$ are expansive, but the $T$-action on $\Sub_{\Q_p^2}$ is not expansive by \Cref{p-adic}\,(2).
\end{example}

\begin{corollary} \label{n-primes} For a locally compact group $G$ and $T\in\Aut(G)$, the following hold:
 \begin{enumerate}
 \item If $G=\Q_{p_1}^{m_1}\times\cdots\times \Q_{p_n}^{m_n}$ for distinct primes $p_1,\ldots, p_n\in\N$ and 
 $m_1,\ldots, m_n\in\N$, and if $T$ is expansive, then $T$ keeps $\Q_{p_i}^{m_i}$ invariant for each $i$ and $G=C(T)\times C(T^{-1})$. 
 Moreover if the $T$-action on $\Sub_G$ is expansive, then $m_1=\cdots =m_n=1$.
 \item If $G=\Q_{p_1}\times\cdots\times \Q_{p_n}$ for distinct primes $p_1,\ldots, p_n\in\N$, then $T$ is expansive if and only if 
 the $T$-action on $\Sub_G$ is expansive. 
 \end{enumerate}
\end{corollary}

\begin{proof} $(1):$ Suppose $T$ is expansive on $G$ as in (1). Then $C(T)C(T^{-1})$ is open in $G$ \cite[Lemma 1.1\,(d)]{GR17}.
Since $G$ is not discrete, either $C(T)$ or $C(T^{-1})$ is nontrivial.  Suppose $C(T)$ is nontrivial. Since $G$ is torsion-free, 
by Theorem B of \cite{GW10}, $C(T)$ is a direct product of $T$-invariant $\Q_{p_i}^{l_i}$, for some prime $p_i$ and $l_i\leq m_i$, 
and hence $C(T)$ is closed. Similarly, either $C(T^{-1})$ is trivial, or it has a similar structure as $C(T)$ described above and it is closed. 
In particular, $C(T)\cap C(T^{-1})=\emptyset$ \cite[Theorem 3.2\,(1$\Leftrightarrow$4)]{BW04}. 
Therefore, $C(T)\times C(T^{-1})$ is an open subgroup of $G$ and it follows that $G=C(T)\times C(T^{-1})$. 
Considering the structure of $C(T)$ and $C(T^{-1}$) described above, it follows that each $\Q_{p_i}^{m_i}$ is $T$-invariant. 
Now suppose the $T$-action on $\Sub_G$ is expansive. 
Then the $T$-action on $\Sub_{\Q_{p_i}^{m_i}}$ is also expansive and by \Cref{p-adic}\,(2), $m_i=1$ for each $i$. 

\noindent $(2):$ Let $G$ be as in $(2)$. The one way implication `if' follows from \Cref{Totally-disconnected}\,$(2)$. Now suppose $T$ is expansive. From (1), 
we get that $\Q_{p_i}$ is $T$-invariant for each $i$, $G=C(T)\times C(T^{-1})$ and any one of the following holds: $C(T)=\{e\}$, $C(T)=G$ or 
$C(T)$ (resp.\ $C(T^{-1})$) is a direct product of some $\Q_{p_i}$. Now the $T$-action on $\Q_{p_i}$ is expansive for each $i$. 

Let $\pi_i:G\to \Q_{p_i}$ be the natural projection and let 
$T_i:\Q_{p_i}\to \Q_{p_i}$ be the  map corresponding to $T$, such that $\pi_i\circ T=T_i\circ\pi_i$, $i=1,\ldots, n$. 
We show that $\Sub_G=\Sub_{\Q_{p_1}}\times\cdots\times\Sub_{\Q_{p_n}}$.
It is easy to see that $\Sub_{\Q_{p_1}}\times\cdots\times\Sub_{\Q_{p_n}}\subset\Sub_G$.

Let $H$ be a closed subgroup of $G$ and let $H_i=\pi_i(H)$ for each $i$. Then $H_i$ is a subgroup of $\Q_{p_i}$ for each $i$. 
We show that $H=H_1\times \cdots\times H_n$. 
This holds if $H$ is trivial. Suppose $H$ is nontrivial. Then not all $H_i$ are trivial.
Note that $H\subset H_1\times \cdots\times H_n$.  
Let $i\in\{1,\ldots, n\}$ be fixed such that $H_i\ne\{0\}$. 
It is enough to show that $H_i\subset H$ (through the canonical inclusion of $\Q_{p_i}$ in $G$). Let $x_i\in H_i$ be nonzero.
There exists $y=(y_1, \ldots, y_n)\in H$ such that $y_i=\pi_i(y)=x_i$.
Let $l=(p_1\cdots p_n)/p_i$. Then $l\in\N$ and for all $j\ne i$, $l^k\to 0$ in $\Q_{p_j}$ as $k\to\infty$ and $l^{k}\in \Z_{p_i}^*$, 
which is a compact multiplicative group. 
It follows that there exists a sequence $\{k_m\}\subset\N$ such that $l^{k_m}\to 1$ in $\Q_{p_i}$ as $m\to\infty$. 
Now $l^{k_m}y=(l^{k_m}y_1,\ldots, l^{k_m}y_n)\to a\in H$, where $a=(a_1,\ldots, a_n)$, $a_i=y_i=x_i$ and $a_j=0$ for all $j\ne i$. 
Therefore $x_i\in H$, and hence $H_i\subset H$. Since this holds for all $i$ such that $H_i\ne\{0\}$ and $\pi_i(H)=H_i$, we get that 
$H=H_1\times\cdots\times H_n$. Now $H_i=H\cap\Q_{p_i}$ is closed for each $i$. 
This implies that $\Sub_G=\Sub_{\Q_{p_1}}\times\cdots\times\Sub_{\Q_{p_n}}$.

If $T$ acts expansively on $G$, then $T_i$ acts expansively on $\Q_{p_i}$ for each $i$.
By \Cref{p-adic}\,(1), the $T_i$-action on $\Sub_{\Q_{p_i}}$ is expansive for each $i$. 
Now the assertion holds by Theorem 2.2.5 of \cite{AH94}.
\end{proof}

Gl\"ockner and Willis have obtained a structure theorem for totally disconnected locally compact contraction groups in \cite{GW10}. 
We use this theorem along with 
\Cref{p-adic} and \Cref{n-primes} and get the following structure theorem for groups $G$ which admit a contractive automorphism 
that acts expansively on $\Sub_G$. Recall that $T$ is said to be contractive if $C(T)=G$. 
Recall also that we call a subgroup $H$ of $G$ $T$-invariant if $T(H)=H$; (in \cite{GW10}, such a group $H$ is called $T$-stable). 

\begin{theorem} \label{main}
Let $G$ be a locally compact group which admits a contractive automorphism $T$. Then $T$ acts expansively on $\Sub_G$ if and only if 
$G=\Q_{p_1}\times\cdots\times\Q_{p_n}$ for distinct primes $p_1,\ldots, p_n$ such that each $\Q_{p_i}$ is $T$-invariant. 
\end{theorem}

\begin{proof} If $T$ is contractive, then $T$ is expansive \cite[Remark 1.10]{GR17}. One way implication `if' follows  from 
Corollary \ref{n-primes}\,$(2)$. (In fact we do not need $\Q_{p_i}$ to be $T$-invariant). Now we assume that $T$ is contractive and 
acts expansively on $\Sub_G$. By Proposition \ref{Totally-disconnected}\,$(1)$, $G$ is totally disconnected. 
By Theorem B of \cite{GW10}, $G=\T\times\,\D$, where $\T$ and $\D$ are closed $T$-invariant groups, $\T$ is the torsion group, 
$\D$ is the group of divisible elements in $G$ and $\D$ is a direct product
of $T$-invariant, nilpotent $p$-adic Lie groups for certain primes $p$, i.e.\ 
$$
\D=G_{p_1}\times\cdots\times G_{p_n}$$
where each $G_p$ is the group of $\Q_p$-rational points of a unipotent linear
algebraic group defined over $\Q_p$, by Theorem 3.5 (ii) of \cite{Wa84}, for $p=p_1,\ldots, p_n$. Suppose $n>1$. 
(We will refer to $G_p$ itself as a unipotent $p$-adic algebraic group). 

We first show that $\T$ is trivial. If possible, suppose $\T$ is nontrivial. Then $\T$ admits a nontrivial closed $T$-invariant subgroup 
$\T_1$ such that $\T_1$ has no proper closed $T$-invariant subgroup and $\T_1\cong F^{(-\N)}\times F^{\N_0}$, a restricted direct product for a 
finite group $F$, with $T|_{\T_1}$ as the right-shift \cite[Theorem 3.3 and Theorem A\,(a)]{GW10}. 
Recall that $\N_0=\N\cup\{0\}$ and $\Z=\N_0\cup -\N$, $F$ is a finite group with the discrete topology, $F^{\N_0}=\prod_{n\in\N_0}F$, 
endowed with the compact product topology. Recall also that $F^{(-\N)} \subset F^{-\N}$ is a subgroup consisting of
 $(x_n)_{n\in-\N}$ with $x_n=1$ for all 
but finitely many $n\in-\N$, where 1 is the identity element of the finite group $F$. Since $\T_1$ is nontrivial, so is $F$.
Let $\HH=\{H\subset F^{\N_0}\mid H\mbox{ is a compact group}\}$. Since the set of subsets of $\N_0$ is uncountable, it implies that $\HH$ is uncountable.
Moreover, as $T((x_n))=((x_{n-1}))$ for all $(x_n)\in \T_1$, we get that $T^n(H)\to\{e\}$ in $\Sub_{\T_1}$ for all $H\in\HH$ \cite[Proposition 2.1]{Wa84}. 
By Theorem 1 of \cite{Re70}, $T|_{\T_1}$ is not expansive. This leads to a contradiction. Therefore, $\T$ is trivial and $G=\D$. 

We show that $\D=\Q_{p_1}\times\cdots\times\Q_{p_n}$.
Let $p$ be a fixed prime in $\{p_1,\ldots,p_n\}$. Then $G_p$ is a $T$-invariant unipotent $p$-adic algebraic group, $T|_{G_p}$ is contractive
and the $T$-action on $\Sub_{{G_p}}$ is expansive. It is enough to show that $G_p=\Q_p$. 

Let $Z$ be the center of $G_p$. Then $Z$ is $T$-invariant and $Z\cong\Q_p^m$ for some $m\in\N$. 
By \Cref{p-adic}\,$(2)$, we get that $Z\cong\Q_p$. If $G_p$ is abelian, then $G_p=\Q_p$. If possible, suppose $G_p$ is not abelian. 
Let $Z^{(1)}$ be the closed subgroup of $G_p$ such that $Z\subset Z^{(1)}$ and $Z^{(1)}/Z$ is the center of $G_p/Z$. Then $Z^{(1)}$ is $T$-invariant and 
it is also a unipotent $p$-adic algebraic subgroup of $G_p$ and, $Z^{(1)}/Z$ is abelian and isomorphic to $\Q_p^{m_1}$, for some $m_1\in\N$. 
By \Cref{subg-q}, the 
$T$-action on $\Sub_{G_p/Z}$, and hence on $\Sub_{Z^{(1)}/Z}$ is expansive. By \Cref{p-adic}\,$(2)$, $Z^{(1)}/Z$ is isomorphic to $\Q_p$. 
Let $x\in Z^{(1)}\setminus Z$. Since $Z^{(1)}$ is unipotent, we have $\{x_t\}_{t\in\Q_p}\subset Z^{(1)}$, the one-parameter subgroup with 
$x=x_1$. As $Z^{(1)}/Z$ is isomorphic to $\Q_p$, we get that $Z^{(1)}=\{x_t\}_{t\in\Q_p}Z$, and it is abelian and isomorphic to $\Q_p^2$. 
Since $Z^{(1)}$ is $T$-invariant, by \Cref{p-adic}\,(2), the $T$-action on $\Sub_{Z^{(1)}}$ is not expansive, which leads to a contradiction. 
Hence $G_p$ is abelian and isomorphic to $\Q_p$.
\end{proof}

\begin{remark} \label{rem3} If a locally compact group $G$ admits an automorphism $T$ which acts expansively on $\Sub_G$, then we have that 
$G$ is totally disconnected, $C(T)C(T^{-1})$ is open in $G$ and  $C(T)$ and $C(T^{-1})$ are closed, and hence locally compact. This implies 
in particular that such $G$ is either discrete, or at least one of $C(T)$ or $C(T^{-1})$ is nontrivial and such a nontrivial 
contraction group is a direct product of finitely many $\Q_{p_i}$ for distinct primes $p_i$ (cf.\ \Cref{main}). It would be interesting to study in detail 
the structure of $C(T)C(T^{-1})$ for such $T$. It would also be interesting to study the structure of (infinite) discrete groups $G$ admitting automorphisms 
which act expansively on $\Sub_G$. 
\end{remark}

\smallskip
\noindent{\bf Acknowledgements:} 
R.\ Shah would like to acknowledge the MATRICS research Grant from DST-SERB (Grant No. MTR/2017/000538), Govt.\ of India. She would also like to acknowledge the support in part by the
International Centre for Theoretical Sciences (ICTS) during a visit for participating in the program Smooth and Homogeneous Dynamics (Code: ICTS/etds2019/09). M.\ B.\ Prajapati would like to acknowledge the UGC-JRF research fellowship from UGC (Grant No. 413318), Govt.\ of India.

\bibliographystyle{amsplain}

\begin{thebibliography}{99}

\bibitem{AH94} Aoki, N., Hiraide, K.: Topological theory of dynamical systems, North-Holland Mathematical Library, 52, North-Holland 
Publishing Co., Amsterdam (1994)

\bibitem{BC16}  Baik, H., Clavier, L.: The space of geometric limits of abelian subgroups of ${\rm PSL}_2(\C)$, Hiroshima Math.\ J.\ 
{\bf 46}, 1--36 (2016)

\bibitem{BW04} Baumgartner, U., Willis, G.\ A.: Contraction groups and scales of automorphisms of totally disconnected locally compact 
groups, Israel J.\ Math.\ {\bf 142}, 221--248 (2004)

\bibitem{BP92} Benedetti, R., Petronio, C.: Lectures on hyperbolic geometry, Universitext, Springer-Verlag, Berlin (1992)

\bibitem{Bi18} Biringer, I.: Metrizing the Chabauty topology, Geom.\ Dedicata {\bf 195}, 19--22 (2018)

\bibitem{BHK09} Bridson, M.\ R.,  de la Harpe, P., Kleptsyn, V.: The Chabauty space of closed subgroups of the three-dimensional 
Heisenberg group, Pacific\ J.\ Math.\ {\bf 240}, 1--48 (2009)

\bibitem{Br60} Bryant, B.\ F.: On expansive homeomorphisms, Pacific J.\ Math.\ {\bf 10}, 1163--1167 (1960)

\bibitem{Ch50} Chabauty, C.: Limite d'ensembles et g\'{e}om\'{e}trie des nombres, Bull.\ Soc.\ Math.\ France {\bf 78}, 143--151 (1950)

\bibitem{Da18} Dani, S.\ G.: Actions of automorphism groups of Lie groups, in Handbook of group actions. Vol. IV, 529--562, 
Adv.\ Lect.\ Math.\ (ALM), 41, Int. Press, Somerville, MA (2018)

\bibitem{Ge18} Gelander, T.: A lecture on invariant random subgroups, in New directions in locally compact groups, 186--204, London Math.\ Soc. 
Lecture Note Ser.\ 447, Cambridge Univ.\ Press, Cambridge (2018)

\bibitem{GR17} Gl\"{o}ckner, H., Raja, C.\ R.\ E.: Expansive automorphisms of totally disconnected, locally compact groups, 
J.\ Group Theory {\bf 20}, 589--619 (2017)

\bibitem{GW10} Gl\"{o}ckner, H., Willis, G.\ A.: Classification of the simple factors appearing in composition series of totally disconnected 
contraction groups, J.\ Reine Angew.\ Math.\ {\bf 643}, 141--169 (2010)

\bibitem{HK} Hamrouni, H., Kadri, B.: On the compact space of closed subgroups of locally compact groups, J.\ Lie Theory {\bf 24}, 715--723 (2014)

\bibitem{KS89} Kitchens, B., Schmidt, K.: Automorphisms of compact groups, Ergodic Theory Dynam.\ Systems {\bf 9}, 691--735 (1989)

\bibitem{Kl09} Kloeckner, B.: The space of closed subgroups of $\R^n$ is stratified and simply connected, J.\ Topol.\ {\bf 2}, 570--588 (2009) 

\bibitem{Mo77} Morris, S.\ A.: Pontryagin duality and the structure of locally compact abelian groups, Cambridge University Press, Cambridge (1977)

\bibitem{PH79} Pourezza, I., Hubbard, J.: The space of closed subgroups of ${\R}\sp{2}$, Topology {\bf 18}, 143--146 (1979)

\bibitem{Re70} Reddy, W.\ L.: Pointwise expansion homeomorphisms, J.\ London Math.\ Soc.\ (2) {\bf 2}, 232--236 (1970)

\bibitem{Sc95} Schmidt, K.: Dynamical systems of algebraic origin, [2011 reprint of the 1995 original], Modern Birkh\"auser Classics, 
Birkh\"auser/Springer Basel AG, Basel (1995)

\bibitem{Sm71} Schochetman, I.: Nets of subgroups and amenability, Proc.\ Amer.\ Math.\ Soc.\ {\bf 29}, 397--403 (1971)

\bibitem{St06} Stroppel, M.: Locally compact groups, EMS Textbooks in Mathematics, European Math.\ Soc.\ (EMS), Z\"urich (2006)

\bibitem{Sh19} Shah, R.: Expansive Automorphisms on locally compact group, New York J. Math. {\bf 26}, 285–302 (2020)

\bibitem{SY19} Shah, R., Yadav, A.\ K.: Distal actions of automorphisms of Lie Groups $G$ on $\Sub_{G}$, arXiv:1909.04397

\bibitem{Ut50} Utz, W.\ R.: Unstable homeomorphisms, Proc.\ Amer.\ Math.\ Soc.\ {\bf 1}, 769--774 (1950)

\bibitem{Wa82} Walters, P.: An introduction to ergodic theory, Grad.\ Texts in Math.\ 79, Springer-Verlag, New York (1982)

\bibitem{Wa84} Wang, J.\ S.\ P.: The Mautner phenomenon for $p$-adic Lie groups, Math.\ Zeit.\ {\bf 185}, 403--412 (1984)

\bibitem{Wi94} Willis, G.: The structure of totally disconnected, locally compact groups, Math.\ Ann.\ {\bf 300}, 341--363 (1994)

\end{thebibliography}

\bigskip\medskip
\noindent{
\advance\baselineskip by 2pt
\begin{tabular}{ll}
Riddhi Shah & \hspace*{3cm}Manoj B. Prajapati \\
School of Physical Sciences & \hspace*{3cm}School of Physical Sciences\\
Jawaharlal Nehru University & \hspace*{3cm}Jawaharlal Nehru University\\
New Delhi 110 067, India & \hspace*{3cm}New Delhi 110 067, India\\
rshah@jnu.ac.in& \hspace*{3cm}manoj.prajapati.2519@gmail.com\\
 riddhi.kausti@gmail.com& \hspace*{3cm}{}\\
\end{tabular}}

\end{document}